\documentclass[a4paper,12pt]{article}
\usepackage{amsmath,amsthm,amssymb}
\usepackage[dvipdfmx]{graphicx}

\numberwithin{equation}{section}

\newtheorem{theorem}{Theorem}
\newtheorem{corollary}[theorem]{Corollary}
\newtheorem{lemma}{Lemma}
\theoremstyle{definition}

\newtheorem{definition}{Definition}

\newcommand{\R}{\mathbb{R}}
\newcommand{\p}{\partial}
\newcommand{\eps}{\varepsilon}

\begin{document}

\begin{center}
\Large{\textbf{
Strong instability of standing waves \\
for nonlinear Schr\"{o}dinger equations \\
with harmonic potential}}
\end{center}

\vspace{2mm}

\begin{center}
Dedicated to Professor Yoshio Tsutsumi on his 60th birthday
\end{center}

\vspace{3mm}

\begin{center}
{\large Masahito Ohta} 
\end{center}
\begin{center}
Department of Mathematics, 
Tokyo University of Science, \\
1-3 Kagurazaka, Shinjuku-ku, Tokyo 162-8601, Japan
\end{center}

\begin{abstract}
We study strong instability of standing waves $e^{i\omega t} \phi_{\omega}(x)$ 
for nonlinear Schr\"{o}dinger equations
with $L^2$-supercritical nonlinearity and a harmonic potential, 
where $\phi_{\omega}$ is a ground state of the corresponding stationary problem. 
We prove that $e^{i\omega t} \phi_{\omega}(x)$ is strongly unstable 
if $\p_{\lambda}^2 E(\phi_{\omega}^{\lambda}) |_{\lambda=1}\le 0$, 
where $E$ is the energy and 
$v^{\lambda}(x)=\lambda^{N/2} v(\lambda x)$ 
is the $L^2$-invariant scaling. 
\end{abstract}

\section{Introduction}\label{sect:intro}

This paper is concerned with the instability of standing waves $e^{i\omega t}\phi_{\omega}(x)$ 
for the nonlinear Schr\"{o}dinger equation with a harmonic potential 
\begin{align}\label{nls}
i\p_t u=-\Delta u+|x|^2u-|u|^{p-1}u,
\quad (t,x)\in \R\times \R^N, 
\end{align}
where $N\ge 1$ and $1<p<2^*-1$. 
Here, $2^*$ is defined by $2^*=2N/(N-2)$ if $N\ge 3$, and $2^*=\infty$ if $N=1,2$. 

It is known that 
for any $\omega\in (-N,\infty)$, 
there exists a unique positive solution (ground state) 
$\phi_{\omega}(x)$ of the stationary problem 
\begin{align}\label{sp}
-\Delta \phi+|x|^2 \phi+\omega \phi-|\phi|^{p-1}\phi=0, 
\quad x\in \R^N
\end{align}
in the energy space 
$$X:=\{v\in H^1(\R^N):  |x| v \in L^2(\R^N)\}.$$ 
Note that the condition $\omega>-N$ appears naturally 
in the existence of positive solutions for \eqref{sp}, 
because the first eigenvalue of $-\Delta +|x|^2$ is $N$. 
For the uniqueness of positive solutions for \eqref{sp}, 
see \cite{HO1,HO2,KT,SW}. 

The Cauchy problem for \eqref{nls} is locally well-posed in the energy space $X$ 
(see \cite[\S 9.2]{caz} and \cite{OhYG}). 
That is, for any $u_0\in X$ there exist $T_{\max}=T_{\max}(u_0)\in (0,\infty]$ 
and a unique solution $u\in C([0,T_{\max}), X)$ 
of \eqref{nls} with intial condition $u(0)=u_0$ such that 
either $T_{\max}=\infty$ (global existence) 
or $T_{\max}<\infty$ and $\displaystyle{\lim_{t\to T_{\max}}\|u(t)\|_{X}=\infty}$ (finite time blowup). 
Moreover, the solution $u(t)$ satisfies the conservations of charge and energy 
\begin{align}\label{conservation}
\|u(t)\|_{L^2}^2=\|u_0\|_{L^2}^2, \quad E(u(t))=E(u_0)
\end{align}
for all $t\in [0,T_{\max})$, 
where the energy $E$ is defined by 
\begin{align} \label{def:E}
E(v)=\frac{1}{2}\|\nabla v\|_{L^2}^2
+\frac{1}{2}\|x v\|_{L^2}^2
-\frac{1}{p+1}\|v\|_{L^{p+1}}^{p+1}. 
\end{align}

Here we give the definitions of stability and instability of standing waves. 

\begin{definition}
We say that the standing wave solution 
$e^{i\omega t}\phi_{\omega}$ of \eqref{nls} is {\em stable} 
if for any $\eps>0$ there exists $\delta>0$ such that 
if $u_0\in X$ and $\|u_0-\phi_{\omega}\|_{X}<\delta$, 
then the solution $u(t)$ of \eqref{nls} with $u(0)=u_0$ exists globally and satisfies 
$$\sup_{t\ge 0}\inf_{\theta\in \R}\|u(t)-e^{i\theta}\phi_{\omega}\|_X<\eps.$$ 
Otherwise, $e^{i\omega t}\phi_{\omega}$ is said to be {\em unstable}. 
\end{definition}

\begin{definition}
We say that $e^{i\omega t}\phi_{\omega}$ is {\em strongly unstable}  
if for any $\eps>0$ there exists $u_0\in X$ such that 
$\|u_0-\phi_{\omega}\|_X<\eps$ and 
the solution $u(t)$ of \eqref{nls} with $u(0)=u_0$ blows up in finite time.
\end{definition}

Before we state our main result, 
we recall some known results 
on the stability and instablity of standing waves $e^{i\omega t}\phi_{\omega}$ for \eqref{nls}. 
When $\omega$ is sufficiently close to $-N$, 
the standing wave solution $e^{i\omega t} \phi_{\omega}$ of \eqref{nls} is stable 
for any $p\in (1,2^*-1)$ (see \cite{FO1}). 
On the other hand, when $\omega$ is sufficiently large, 
the standing wave solution $e^{i\omega t} \phi_{\omega}$ of \eqref{nls} is 
stable if $1<p\le 1+4/N$ (see \cite{FO1,fuk05}), 
and unstable if $1+4/N<p<2^*-1$ (see \cite{FO2}). 
More precisely, it is proved in \cite{FO2} that 
$e^{i\omega t} \phi_{\omega}$ is unstable 
if $\p_{\lambda}^2 E(\phi_{\omega}^{\lambda}) |_{\lambda=1}<0$, 
where $v^{\lambda}(x)=\lambda^{N/2} v(\lambda x)$ 
is the $L^2$-invariant scaling 
(see also \cite{oht95}). 

However, the strong instability of $e^{i\omega t}\phi_{\omega}$ has been unknown for \eqref{nls}, 
although there are some results on blowup  
(see, e.g., \cite{carles02, zhang05,wang}). 

Now we state our main result in this paper. 

\begin{theorem}\label{thm1}
Let $N\ge 1$, $1+4/N<p<2^*-1$, $\omega>-N$, 
and let $\phi_{\omega}$ be the positive solution of \eqref{sp}. 
Assume that $\p_{\lambda}^2 E(\phi_{\omega}^{\lambda}) |_{\lambda=1}\le 0$. 
Then, the standing wave solution $e^{i\omega t}\phi_{\omega}$ of \eqref{nls} 
is strongly unstable. 
\end{theorem}

We remark that 
by the scaling $v^{\lambda}(x)=\lambda^{N/2}v(\lambda x)$
for $\lambda>0$, 
we have $\|v^{\lambda}\|_{L^2}^2=\|v\|_{L^2}^2$ and 
\begin{align}\label{Ed0} 
E(v^{\lambda})
=\frac{\lambda^2}{2}\|\nabla v\|_{L^2}^2
+\frac{\lambda^{-2}}{2}\|x v\|_{L^2}^2
-\frac{\lambda^{\alpha}}{p+1} \|v\|_{L^{p+1}}^{p+1}.
\end{align}
Here and hereafter, we put 
$$\alpha:=\dfrac{N}{2}(p-1)>2.$$

Moreover, we define 
$S_{\omega}(v)=E(v)+\dfrac{\omega}{2}\|v\|_{L^2}^2$ 
for $v\in X$. 
Then, $\phi_{\omega}$ satisfies $S_{\omega}'(\phi_{\omega})=0$, and 
\begin{align*}
0=\p_{\lambda} S_{\omega}(\phi_{\omega}^{\lambda}) \big|_{\lambda=1}
&=\|\nabla \phi_{\omega}\|_{L^2}^2-\|x \phi_{\omega} \|_{L^2}^2
-\frac{\alpha}{p+1} \|\phi_{\omega}\|_{L^{p+1}}^{p+1}, \\
\p_{\lambda}^2 E(\phi_{\omega}^{\lambda}) \big|_{\lambda=1}
&=\|\nabla \phi_{\omega}\|_{L^2}^2+3\|x \phi_{\omega} \|_{L^2}^2
-\frac{\alpha (\alpha-1)}{p+1} \|\phi_{\omega}\|_{L^{p+1}}^{p+1} \\
&=4 \|x \phi_{\omega} \|_{L^2}^2
-\frac{\alpha (\alpha-2)}{p+1} \|\phi_{\omega}\|_{L^{p+1}}^{p+1}.
\end{align*}
Thus, the condition $\p_{\lambda}^2 E(\phi_{\omega}^{\lambda}) |_{\lambda=1}\le 0$ 
is equivalent to 
$$\frac{ \|x \phi_{\omega} \|_{L^2}^2}{\|\phi_{\omega}\|_{L^{p+1}}^{p+1}}
\le \frac{\alpha (\alpha-2)}{4(p+1)}.$$
Furthermore, it is proved in Section 2 of \cite{FO2} that 
\begin{equation*}
\lim_{\omega \to \infty} 
\frac{ \|x \phi_{\omega} \|_{L^2}^2}{\|\phi_{\omega}\|_{L^{p+1}}^{p+1}}=0.
\end{equation*}

Therefore, as a corollary of Theorem \ref{thm1}, 
we have the following. 

\begin{corollary}\label{cor2}
Let $N\ge 1$, $1+4/N<p<2^*-1$, $\omega>-N$, 
and let $\phi_{\omega}$ be the positive solution of \eqref{sp}. 
Then, there exists $\omega_0 \in (-N,\infty)$ 
depending only on $N$ and $p$ such that 
the standing wave solution $e^{i\omega t}\phi_{\omega}$ of \eqref{nls} 
is strongly unstable for all $\omega\in (\omega_0,\infty)$. 
\end{corollary}

We give the proof of Theorem \ref{thm1} in Sections 2 and 3. 
In Section 2, we introduce a subset $\mathcal{B}_{\omega}$ of $X$, 
and Theorem \ref{thm1} is reduced to Theorem \ref{thm3}. 
In Section 3, we give the proof of Theorem \ref{thm3}. 
The key to the proof of Theorem \ref{thm3} is Lemma \ref{lem-key}. 
The proof of Lemma \ref{lem-key} relies heavily on the specialty of harmonic potential, 
and it is not applicable to other types of nonlinear Schr\"odinger equations. 
For example, consider the nonlinear Schr\"odinger equation 
with a delta potential in one space dimension 
\begin{equation} \label{nls-delta}
i\p_t u=-\p_x^2 u-\gamma \delta (x) u-|u|^{p-1}u, \quad 
(t,x)\in \R\times \R,
\end{equation}
where $\gamma>0$, 
$\delta (x)$ is the delta measure at the origin, 
$5<p<\infty$, 
$\omega>\gamma^2/4$, 
and $\phi_{\omega}$ is a unique positive solution 
of the corresponding stationary problem. 
It is proved in \cite{OY2} that 
if $E_{\gamma} (\phi_{\omega})>0$, 
then $e^{i\omega t} \phi_{\omega}$ is strongly unstable for \eqref{nls-delta}
(see also \cite{nak, OY1} for related results),  
where the energy for \eqref{nls-delta} is defined by 
\begin{align*}
E_{\gamma} (v)= \frac{1}{2} \|\p_x v\|_{L^2}^2
-\frac{\gamma}{2} |v(0)|^2 
-\frac{1}{p+1}\|v\|_{L^{p+1}}^{p+1}, \quad 
v\in H^1(\R). 
\end{align*}
Here, we remark that $E_{\gamma} (\phi_{\omega})>0$ 
implies $\p_{\lambda}^2 E_{\gamma} (\phi_{\omega}^{\lambda}) |_{\lambda=1}<0$ 
for this case. 
The problem for \eqref{nls-delta} 
is completely different from that for \eqref{nls}, 
and the proof of Lemma \ref{lem-key} in this paper 
is not applicable to \eqref{nls-delta}. 
So, it is still an open problem 
whether the standing wave solution 
$e^{i\omega t} \phi_{\omega}$ of \eqref{nls-delta} is strongly unstable or not 
for the case where 
$E_{\gamma}(\phi_{\omega})\le 0$ and 
$\p_{\lambda}^2 E_{\gamma}(\phi_{\omega}^{\lambda}) |_{\lambda=1}<0$
(see \cite[\S 4]{OY2} for more remarks). 

\section{Proof of Theorem \ref{thm1}}

Throughout this section, we assume that 
$1+4/N<p<2^*-1$, $\omega>-N$, 
and $\phi_{\omega}$ is the positive solution of \eqref{sp}. 
We put $\alpha=N(p-1)/2>2$. 

The proofs of blowup 
and strong instability of standing waves for nonlinear Schr\"odinger equations 
rely on the virial identity 
(see, e.g., \cite{BC, caz, lec, zhang02}). 
Let $u(t)$ be the solution of \eqref{nls} with $u(0)=u_0\in X$. 
Then, the function $t\mapsto \|xu(t)\|_{L^2}^2$ is in $C^2[0,T_{\max})$, and satisfies 
\begin{align}
\frac{d^2}{dt^2}\|xu(t)\|_{L^2}^2=16 P(u(t))
\label{virial}
\end{align}
for all $t\in [0,T_{\max})$, where 
\begin{align}\label{def:P}
P(v)= \frac{1}{2} \|\nabla v\|_{L^2}^2
- \frac{1}{2} \|x v\|_{L^2}^2
-\frac{\alpha}{2(p+1)}\|v\|_{L^{p+1}}^{p+1}. 
\end{align}

Moreover, we define 
\begin{align}
R(v)=\|\nabla v\|_{L^2}^2
+3\|x v\|_{L^2}^2
-\frac{\alpha (\alpha-1)}{p+1}\|v\|_{L^{p+1}}^{p+1}.
\label{def:R}
\end{align}
Note that by \eqref{Ed0}, we have 
\begin{align}\label{PR}
P(v^{\lambda})= \frac{1}{2} \lambda \p_{\lambda} E(v^{\lambda}), \quad 
R(v^{\lambda})=\lambda^2 \p_{\lambda}^2 E(v^{\lambda}) 
\end{align}
for $\lambda>0$. 
We also define 
\begin{align*}
&\mathcal{A}_{\omega}
=\left\{v\in X: 
E(v)<E(\phi_{\omega}), ~ 
\|v\|_{L^2}^2=\|\phi_{\omega}\|_{L^2}^2, ~
\|v\|_{L^{p+1}}^{p+1}>\|\phi_{\omega}\|_{L^{p+1}}^{p+1} \right\}, \\
&\mathcal{B}_{\omega}
=\left\{ v\in \mathcal{A}_{\omega}: P(v)<0 \right\}. 
\end{align*} 

\begin{lemma} \label{lem-bom}
Assume that $R(\phi_{\omega})\le 0$. 
Then, $\phi_{\omega}^{\lambda}\in \mathcal{B}_{\omega}$ 
for all $\lambda>1$. 
\end{lemma}

\begin{proof}
First, we have 
$\|\phi_{\omega}^{\lambda}\|_{L^2}^2=\|\phi_{\omega}\|_{L^2}^2$ and 
$\|\phi_{\omega}^{\lambda}\|_{L^{p+1}}^{p+1}
=\lambda^{\alpha} \|\phi_{\omega}\|_{L^{p+1}}^{p+1}
>\|\phi_{\omega}\|_{L^{p+1}}^{p+1}$ 
for $\lambda>1$. 

Next, by drawing the graphs of the functions  
$\lambda\mapsto P(\phi_{\omega}^{\lambda})$ and 
$\lambda\mapsto E(\phi_{\omega}^{\lambda})$ 
using \eqref{Ed0}, \eqref{PR}, $P(\phi_{\omega})=0$ and $R(\phi_{\omega})\le 0$, 
we see that 
$P(\phi_{\omega}^{\lambda})<P(\phi_{\omega})=0$ and 
$E(\phi_{\omega}^{\lambda})<E(\phi_{\omega})$ for $\lambda>1$. 
This completes the proof. 
\end{proof}

Since $\phi_{\omega}^{\lambda}\to \phi_{\omega}$ in $X$ as $\lambda\to 1$, 
Theorem \ref{thm1} follows from Lemma \ref{lem-bom} and the following Theorem \ref{thm3}. 

\begin{theorem}\label{thm3}
Let $N\ge 1$, $1+4/N<p<2^*-1$, $\omega>-N$, 
and assume that $R(\phi_{\omega})\le 0$. 
If $u_0\in \mathcal{B}_{\omega}$, 
then the solution $u(t)$ of \eqref{nls} with $u(0)=u_0$ blows up in finite time. 
\end{theorem}

We give the proof of Theorem \ref{thm3} in the next section. 

\section{Proof of Theorem \ref{thm3}}

Throughout this section, we assume that 
$1+4/N<p<2^*-1$, $\omega>-N$, 
and $\phi_{\omega}$ is the positive solution of \eqref{sp}. 

\begin{lemma} \label{lem-VarCha}
If $v\in X$ satisfies 
$\|v\|_{L^2}^2=\|\phi_{\omega}\|_{L^2}^2$ and 
$\|v\|_{L^{p+1}}^{p+1}=\|\phi_{\omega}\|_{L^{p+1}}^{p+1}$, 
then $E(\phi_{\omega})\le E(v)$. 
\end{lemma}

\begin{proof}
It is well known that 
$$S_{\omega}(\phi_{\omega})
=\inf \left\{S_{\omega}(v):
v\in X, ~ \|v\|_{L^{p+1}}^{p+1}=\|\phi_{\omega}\|_{L^{p+1}}^{p+1} \right\}$$
(see, e.g., Lemma 3.1 of \cite{FO2}), where 
$S_{\omega}(v)=E(v)+\dfrac{\omega}{2}\|v\|_{L^2}^2$. 

Thus, 
if $v\in X$ satisfies 
$\|v\|_{L^2}^2=\|\phi_{\omega}\|_{L^2}^2$ and 
$\|v\|_{L^{p+1}}^{p+1}=\|\phi_{\omega}\|_{L^{p+1}}^{p+1}$, 
then 
\begin{equation*}
E(\phi_{\omega})
=S_{\omega} (\phi_{\omega})-\frac{\omega}{2} \|\phi_{\omega}\|_{L^{2}}^{2}
\le S_{\omega} (v)-\frac{\omega}{2} \|v\|_{L^{2}}^{2} =E(v). 
\end{equation*}
This completes the proof. 
\end{proof}

\begin{lemma}\label{lem-invA}
The set $\mathcal{A}_{\omega}$ is invariant under the flow of \eqref{nls}. 
That is, if $u_0 \in \mathcal{A}_{\omega}$, 
then the solution $u(t)$ of \eqref{nls} with $u(0)=u_0$ satisfies 
$u(t)\in \mathcal{A}_{\omega}$ for all $t\in [0,T_{\max})$. 
\end{lemma}

\begin{proof}
This follows from the conservation laws \eqref{conservation} and Lemma \ref{lem-VarCha}. 
\end{proof}

By \eqref{def:E}, \eqref{def:P} and \eqref{def:R}, 
we have 
\begin{align}
&E(v)-P(v)
=\|x v\|_{L^2}^2
+\frac{\alpha-2}{2(p+1)}\|v\|_{L^{p+1}}^{p+1}, 
\label{EP} \\
&R(v)-2P(v)
=4\|x v\|_{L^2}^2
-\frac{\alpha (\alpha-2)}{p+1}\|v\|_{L^{p+1}}^{p+1}.
\label{RP} 
\end{align}

The following lemma is the key to the proof of Theorem \ref{thm3}. 

\begin{lemma} \label{lem-key}
Assume that $R(\phi_{\omega})\le 0$. 
If $v\in X$ satisfies 
\begin{equation}\label{key1}
P(v)\le 0,~ 
\|v\|_{L^2}^2=\|\phi_{\omega} \|_{L^2}^2, ~
\|v\|_{L^{p+1}}^{p+1}>\|\phi_{\omega}\|_{L^{p+1}}^{p+1}, 
\end{equation}
then $E(\phi_{\omega})\le E(v)-P(v)$. 
\end{lemma}

\begin{proof}
Let $v\in X$ satisfy \eqref{key1}. 

If $\|x v\|_{L^2}^2 \ge \|x\phi_{\omega}\|_{L^{2}}^{2}$, 
then it follows from 
$P(\phi_{\omega})=0$ and \eqref{EP} that 
\begin{align*}
E(\phi_{\omega})
&=\|x \phi_{\omega}\|_{L^2}^2
+\frac{\alpha-2}{2(p+1)}\|\phi_{\omega}\|_{L^{p+1}}^{p+1} \\
&\le \|x v\|_{L^2}^2
+\frac{\alpha-2}{2(p+1)}\|v\|_{L^{p+1}}^{p+1}
=E(v)-P(v).
\end{align*}

So, in what follows, 
we assume that 
\begin{equation}\label{key1.5}
\|x v\|_{L^2}^2\le \|x\phi_{\omega}\|_{L^{2}}^{2}.
\end{equation}

Let 
\begin{equation}\label{key2}
\lambda_0
=\left(\frac{\|\phi_{\omega}\|_{L^{p+1}}^{p+1}}{\|v\|_{L^{p+1}}^{p+1}}\right)^{1/\alpha}.
\end{equation}
Then, $0<\lambda_0<1$, 
$\|v^{\lambda_0}\|_{L^{p+1}}^{p+1}=\|\phi_{\omega}\|_{L^{p+1}}^{p+1}$, 
$\|v^{\lambda_0}\|_{L^{2}}^{2}=\|\phi_{\omega}\|_{L^{2}}^{2}$, 
and it follows from Lemma \ref{lem-VarCha} that 
\begin{equation} \label{key2.5} 
E(\phi_{\omega}) \le E(v^{\lambda_0}). 
\end{equation}

Next, we define 
\begin{align*}
f(\lambda)
:=E(v^{\lambda})- \lambda^2 P(v)
=\frac{\lambda^2+\lambda^{-2}}{2} \|x v\|_{L^2}^2
+\frac{\alpha \lambda^2-2 \lambda^{\alpha}}{2(p+1)} 
\|v\|_{L^{p+1}}^{p+1}
\end{align*}
for $\lambda>0$.  
Then, $f(\lambda_0)\le f(1)$ 
if and only if 
\begin{equation}\label{key3}
(p+1)(\lambda_0^2+\lambda_0^{-2}-2) \|x v\|_{L^2}^2
\le (2 \lambda_0^{\alpha}-\alpha \lambda_0^2+\alpha-2)
\|v\|_{L^{p+1}}^{p+1}. 
\end{equation}

Here, since $P(\phi_{\omega})=0$ and $R(\phi_{\omega})\le 0$, 
it follows from \eqref{RP} that 
\begin{equation}\label{key4}
4(p+1) \|x \phi_{\omega} \|_{L^2}^2
\le \alpha(\alpha-2) \|\phi_{\omega}\|_{L^{p+1}}^{p+1}. 
\end{equation}
Thus, by \eqref{key1.5}, \eqref{key4} and \eqref{key2}, we have 
\begin{align*}
&4(p+1) (\lambda_0^2+\lambda_0^{-2}-2) \|x v\|_{L^2}^2 
=4(p+1) (\lambda_0-\lambda_0^{-1})^2 \|x v\|_{L^2}^2 \\
&\le 4(p+1) (\lambda_0-\lambda_0^{-1})^2 \|x \phi_{\omega} \|_{L^2}^2 \\
&\le \alpha(\alpha-2) (\lambda_0-\lambda_0^{-1})^2 \|\phi_{\omega} \|_{L^{p+1}}^{p+1} 
=\alpha(\alpha-2) (\lambda_0-\lambda_0^{-1})^2 \lambda_0^{\alpha} \|v\|_{L^{p+1}}^{p+1}.
\end{align*}
Therefore, \eqref{key3} holds if 
\begin{equation}\label{key5}
\alpha(\alpha-2) (\lambda_0-\lambda_0^{-1})^2 \lambda_0^{\alpha} 
\le 4 (2 \lambda_0^{\alpha}-\alpha \lambda_0^2+\alpha-2).
\end{equation}

Here, we put $\beta=\alpha/2$ and define 
$$g(s):=s^{\beta}-1-\beta (s-1) -\frac{\beta (\beta-1)}{2}(s-1)^2 s^{\beta-1}$$
for $s>0$. 
Then, \eqref{key5} is equivalent to $g(\lambda_0^2)\ge 0$. 
By the Taylor expansion of $s^{\beta}$ at $s=1$, we have 
$$g(\lambda_0^2)
=\frac{\beta (\beta-1)}{2}(\lambda_0^2-1)^2
\left\{ \xi^{\beta-2}-(\lambda_0^2)^{\beta-1} \right\}$$ 
for some $\xi\in (\lambda_0^2,1)$. 
Since $\beta>1$ and $\lambda_0^2<\xi<1$, we have 
$$(\lambda_0^2)^{\beta-1} 
\le \xi^{\beta-1} \le \xi^{\beta-2},$$ 
and obtain $g(\lambda_0^2)\ge 0$. 
Thus, we have \eqref{key3} and $f(\lambda_0)\le f(1)$. 

Finally, since $P(v)\le 0$, 
it follows from \eqref{key2.5}  that 
\begin{align*}
E(\phi_{\omega}) \le E(v^{\lambda_0})
\le E(v^{\lambda_0})- \lambda_0^2 P(v)
=f(\lambda_0)
\le f(1)=E(v)-P(v).
\end{align*}
This completes the proof. 
\end{proof}

\begin{lemma}\label{lem-invB}
Assume that $R(\phi_{\omega})\le 0$. 
Then, the set $\mathcal{B}_{\omega}$ is invariant under the flow of \eqref{nls}. 
\end{lemma}

\begin{proof}
Let $u_0\in \mathcal{B}_{\omega}$ 
and $u(t)$ be the solution of \eqref{nls} with $u(0)=u_0$. 
Since $\mathcal{A}_{\omega}$ is invariant under the flow of \eqref{nls}, 
we have only to show that 
$P(u(t))<0$ for all $t\in [0,T_{\max})$. 

Suppose that there exists $t_1\in (0,T_{\max})$ such that $P(u(t_1))\ge 0$. 
Then, by the continuity of the function $t\mapsto P(u(t))$, 
there exists $t_0\in (0,t_1]$ such that $P(u(t_0))=0$. 
Moreover, by Lemma \ref{lem-invA}, 
we have 
$\|u(t_0)\|_{L^2}^2=\|\phi_{\omega}\|_{L^2}^2$ and 
$\|u(t_0)\|_{L^{p+1}}^{p+1}>\|\phi_{\omega}\|_{L^{p+1}}^{p+1}$.  
Thus, by Lemma \ref{lem-key}, 
we have 
$$E(\phi_{\omega}) 
\le E(u(t_0))-P(u(t_0))=E(u(t_0)).$$

On the other hand, 
since 
it follows from Lemma \ref{lem-invA} that 
$E(u(t_0))<E(\phi_{\omega})$, 
this is a contradiction. 

Therefore, 
$P(u(t))<0$ for all $t\in [0,T_{\max})$.  
\end{proof}

Now we give the proof of Theorem \ref{thm3}. 

\vspace{2mm} \noindent 
{\bf Proof of Theorem \ref{thm3}}. 
Let $u_0\in \mathcal{B}_{\omega}$ 
and let $u(t)$ be the solution of \eqref{nls} with $u(0)=u_0$. 
Then, by Lemma \ref{lem-invB}, 
$u(t)\in \mathcal{B}_{\omega}$ for all $t\in [0,T_{\max})$. 

Moreover, by the virial identity \eqref{virial}, Lemma \ref{lem-key}
and the conservation of energy \eqref{conservation}, we have 
\begin{align*}
\frac{1}{16}\frac{d^2}{dt^2}\|xu(t)\|_{L^2}^2
=P(u(t))\le E(u(t))-E(\phi_{\omega})
=E(u_0)-E(\phi_{\omega})<0
\end{align*}
for all $t\in [0,T_{\max})$, 
which implies $T_{\max}<\infty$. 
This completes the proof. 
\qed \vspace{2mm}

\vspace{2mm} \noindent 
{\bf Acknowledgments}. 
This work was supported by JSPS KAKENHI Grant Numbers 15K04968 and 26247013. 
The author thanks Takahiro Yamaguchi for valuable discussions.

\end{document}